\documentclass[twoside]{amsart}
\usepackage[T1]{fontenc}
\usepackage[utf8]{inputenc}

\usepackage{hyperref}
\usepackage{bbm}
\usepackage[slantedGreeks, partialup, noDcommand]{kpfonts}
\usepackage{graphicx}

\newtheorem{lemma}{Lemma}[section]
\newtheorem{theorem}{Theorem}[section]

\newtheorem{remark}{Remark}[section]

\title[Hypercontractivity of the semigroup of the fractional laplacian on the $n$-sphere]{Hypercontractivity of the semigroup\\ of the fractional laplacian on the $n$-sphere}

\author{Rupert L. Frank}
\address[R. Frank]{Mathematics 253-37, Caltech, Pasa\-de\-na, CA 91125, USA, and Mathematisches Institut, Ludwig-Maximilans Univers\"at M\"unchen, The\-resienstr. 39, 80333 M\"unchen, Germany}
\email{rlfrank@caltech.edu}

\author{Paata Ivanisvili}
\address[P. Ivanisvili]{Department of Mathematics,
North Carolina State University,
Raleigh, NC 27695}
\email{pivanis@ncsu.edu}

\begin{document}


\begin{abstract}
For $1<p\leq q$ we show that the Poisson semigroup $e^{-t\sqrt{-\Delta}}$  on  the $n$-sphere is hypercontractive from $L^{p}$ to $L^{q}$  in dimensions $n \leq  3$ if and only if $e^{-t\sqrt{n}} \leq \sqrt{\frac{p-1}{q-1}}$. We also show that the equivalence fails in large dimensions. 
\end{abstract}

\keywords{Hypercontractivity, Poisson Semigroup, n-sphere}
\subjclass[2010]{39B62, 42B35, 47A30}

 \renewcommand{\thefootnote}{${}$} \footnotetext{\copyright\, 2021 by the author. This paper may be reproduced, in its entirety, for non-commercial purposes.}

\maketitle

\section{Introduction}
\subsection{Poisson semigroup on the sphere}
Let 
\begin{align*}
\mathbb{S}^{n} = \{ x \in \mathbb{R}^{n+1}\, :\, \|x\| =1\}
\end{align*}
 be the unit sphere in $\mathbb{R}^{n+1}$, where $\|x\| = \sqrt{x_{1}^{2}+\ldots+x_{n+1}^{2}}$ for $x = (x_{1}, \ldots, x_{n+1}) \in \mathbb{R}^{n+1}$.  Let $\Delta$ be the Laplace--Beltrami operator on $\mathbb S^n$. We will be working with spherical polynomials $f : \mathbb{S}^{n} \to \mathbb{C}$, i.e., finite sums 
 $$
 f(\xi) = \sum_{d\geq 0} H_{d}(\xi),
 $$
 where $H_{d}$ satisfies 
 $$
 \Delta H_{d} = -d(d+n-1) H_{d}.
 $$
  The heat semigroup $e^{t\Delta}$ is defined by $e^{t\Delta} f = \sum_{d\geq 0} e^{-d(d+n-1)t}H_{d}$. The hypercontractivity result for the heat semigroup  on $\mathbb{S}^{n}$ states that for any $1\leq p\leq q<\infty$, any integer $n \geq 1$, and any $t\geq 0$ we have 
 \begin{align}\label{heat}
 \| e^{t\Delta} f\|_{q} \leq \|f\|_{p} \quad \text{for all}\ f 
 \qquad \text{if and only if} \qquad e^{-tn} \leq \sqrt{\frac{p-1}{q-1}},
 \end{align}
 where $\|f\|_{p}^{p} =\|f\|_{L^{p}(\mathbb{S}^{n}, d\sigma_{n})}^{p}= \int_{\mathbb{S}^{n}} |f|^{p} d\sigma_{n}$, and  $d\sigma_{n}$ is the normalized surface area measure of $\mathbb{S}^{n}$. The case $n=1$ was solved independently in \cite{R1} and \cite{W1}, and the general case $n\geq 2$ was settled in \cite{MW1}. We remark that the condition $e^{-tn} \leq \sqrt{\frac{p-1}{q-1}}$ in (\ref{heat}) is different from the classical  hypercontractivity condition  $e^{-t} \leq \sqrt{\frac{p-1}{q-1}}$  in Gauss space due to Nelson~\cite{N1}, and on the hypercube due to Bonami~\cite{Bo1}. The appearance of the extra factor $n$ in (\ref{heat}) can be explained from the fact that the spectral gap (the smallest nonzero eigenvalue) of $-\Delta$ equals $n$. 
 
In \cite{MW1} the authors ask what the corresponding hypercontractivity estimates are   for the Poisson semigroup  on $\mathbb{S}^{n}$. As pointed out in \cite{MW1}, there are two natural  Poisson semigroups  on $\mathbb{S}^{n}$ one can consider: 1)  $e^{-t\sqrt{-\Delta}}f$, and  2)   $P_{r}f = \sum r^{d} H_{d}$, $r \in [0,1]$. Notice that when $n=1$ both of these  semigroups coincide (with $r = e^{-t}$). It was conjectured by E.~Stein that 
$$
\| P_{r} f\|_{q} \leq \|f\|_{p} \quad \text{if and only if} \quad r \leq \sqrt{\frac{p-1}{q-1}}
$$
holds on $\mathbb{S}^{n}$ for all $n\geq 1$. Besides the case $n=1$ mentioned above, the case $n=2$ was  confirmed in \cite{Ja1}, and the general case $n\geq 2$ in \cite{Be1}. 

The question of hypercontractivity for the semigroup $e^{-t\sqrt{-\Delta}}$ on $\mathbb{S}^{n}$ for $n\geq 2$, however, has remained open. Since the spectral gap of $\sqrt{-\Delta}$ equals $\sqrt{n}$, it is easy to see that a necessary condition for the estimate $\| e^{-t\sqrt{-\Delta}} f\|_{q} \leq \|f\|_{p}$ is $e^{-t\sqrt{n}} \leq \sqrt{\frac{p-1}{q-1}}$; see Section~\ref{auc}. One might conjecture that this necessary condition is also sufficient. Surprisingly, it turns out the answer is positive in small dimensions and negative in large dimensions. 
\begin{theorem}\label{mth}
Let $1<p<q$, $n\geq 1$, and $t\geq 0$. Then  
\begin{align}\label{pois}
\textup{(i)}\; \;   \| e^{-t\sqrt{-\Delta}} f\|_{q} \leq \|f\|_{p} \quad \text{for all}\ f
\qquad \text{implies} \qquad \textup{(ii)}\;\; e^{-t\sqrt{n}} \leq \sqrt{\frac{p-1}{q-1}}.
\end{align}
Moreover, $(ii)$ implies $(i)$ in dimensions $n \leq 3$. Finally, for any $q>\max\{2,p\}$, there exists $n_{0}=n_{0}(p,q)\geq 4$ such that $\textup{(ii)}$ does not imply $\textup{(i)}$ in dimensions $n$ with  $n\geq n_{0}$.
\end{theorem} 

 It remains an open  problem to  find a necessary and sufficient condition on $t> 0$ in dimensions $n\geq 4$ for which the semigroup $e^{-t\sqrt{-\Delta}}$ is hypercontractive from $L^{p}(\mathbb{S}^{n})$ to $L^{q}(\mathbb{S}^{n})$. 

\section{Proof of Theorem~\ref{mth}}

\subsection{The necessity part $\textup{(i)} \Rightarrow \textup{(ii)}$}\label{auc} 
We recall this standard argument for the sake of completeness. Let $f(\xi) = 1+\varepsilon H_{1}(\xi)$ where $H_{1}$ is any (real) spherical harmonic of degree $1$, i.e., $\Delta H_{1} = -n H_{1}$. Then $e^{-t\sqrt{-\Delta}}f(\xi) = 1+ \varepsilon e^{-t\sqrt{n} } H_{1}(\xi)$. As $\varepsilon \to 0$, we obtain 
\begin{align*}
&\int_{\mathbb{S}^{n}} | 1+ \varepsilon e^{-t\sqrt{n} } H_{1}(\xi)|^{q} d\sigma_{n} \\
& = \int_{\mathbb{S}^{n}} \left( 1 + q \varepsilon e^{-t \sqrt{n} } H_{1}(\xi) + \frac{q(q-1)}{2}\varepsilon^{2} e^{-2t\sqrt{n} } H^{2}_{1}(\xi) + O(\varepsilon^{3}) \right) d\sigma_{n} \\
& = 1 + \frac{q(q-1)}{2} \varepsilon^{2} e^{-2t\sqrt{n}} \|H_{1}\|_{2}^{2} + O(\varepsilon^{3}).
\end{align*} 
Thus, 
\begin{align}\label{tayheat}
\| e^{-t\sqrt{-\Delta}} f\|_{q} = 1 + \frac{q-1}{2} \varepsilon^{2} e^{-2t\sqrt{n}} \|H_{1}\|_{2}^{2} + O(\varepsilon^{3}).
\end{align}
Similarly, we have 
\begin{align}\label{tay}
\|f\|_{p} = 1 + \frac{p-1}{2} \varepsilon^{2}  \|H_{1}\|_{2}^{2} + O(\varepsilon^{2}).
\end{align}
Substituting (\ref{tayheat}) and (\ref{tay}) into the inequality $\| e^{-t\sqrt{-\Delta}} f\|_{q} \leq  \|f\|_{p}$, and taking $\varepsilon \to 0$ we obtain the necessary condition $e^{-2t\sqrt{n}} \leq \frac{p-1}{q-1}$ which coincides with $\textup{(ii)}$ in  (\ref{pois}).

\subsection{The sufficiency part $\textup{(ii)}  \Rightarrow \textup{(i)}$ in dimensions $n=1,2,3$.}

Our goal is to show that if $1<p < q$ and if $t\geq 0$ is such that $e^{-t2\sqrt{n}} \leq \frac{p-1}{q-1}$, then 
\begin{align}\label{hyp22}
\| e^{-t\sqrt{-\Delta}} f\|_{q} \leq \|f\|_{p} \quad \text{in dimensions} \quad n=1,2,3. 
\end{align}
The case $n=1$ was confirmed in \cite{W1}. In what follows we assume $n \in \{2,3\}$.
First we need the fact that the heat semigroup $e^{t \Delta}$ has a nonnegative kernel. Indeed, for each $t>0$ there exists $K_{t} : [-1,1] \to [0, \infty)$ such that 
$$
e^{t\Delta}f(\xi) = \int_{\mathbb{S}^{n}} K_{t}(\xi \cdot \eta) f(\eta) d\sigma_{n}(\eta),
$$
where $\xi \cdot \eta = \sum_{j=1}^{n+1} \xi_{j} \eta_{j}$ for $\xi = (\xi_{1}, \ldots, \xi_{n+1})$ and $\eta = (\eta_{1}, \ldots, \eta_{n+1})$, see, for example, Proposition 4.1 in \cite{MW1}.  Next, we recall the subordination formula 
\begin{align}\label{sub}
e^{-x} = \frac1{\sqrt\pi} \int_{0}^{\infty} e^{-y-x^{2}/(4y)} \frac{dy}{\sqrt y}  \quad \text{valid for all}  \ x \geq 0,
\end{align}
By the functional calculus, we deduce that the Poisson semigroup $e^{-t\sqrt{-\Delta}}$ has a positive kernel with total mass $1$. The latter fact together with the convexity of the map $x \mapsto |x|^{p}$ for $p\geq 1$ implies that $\| e^{-t\sqrt{-\Delta}}\|_{p}\leq \|f\|_{p}$ for all $t \geq 0$. Thus, it suffices to verify (\ref{hyp22}) for those $t\geq 0$ for which $e^{-2t\sqrt{n}} = \frac{p-1}{q-1}$. 

Next we claim that it suffices to verify (\ref{hyp22})  only for the powers $p,q$ such that $2\leq p \leq q$. Indeed, assume (\ref{hyp22}) holds  for $2\leq p\leq q$. By duality and the symmetry of the semigroup $e^{-t \sqrt{-\Delta}}$ we obtain $\| e^{-t\sqrt{-\Delta}} f\|_{p'} \leq \|f\|_{q'}$ where $p' = \frac{p}{p-1}$, $q' = \frac{q}{q-1}$, $1<q' \leq p'\leq 2$. Notice that $\frac{p-1}{q-1} = \frac{q'-1}{p'-1}$, thus we  extend (\ref{hyp22}) to all $p,q$ such that $1<p\leq q \leq 2$. It remains to extend (\ref{hyp22}) for those powers $p,q$ when $p\leq 2 \leq q$. To do so, let $p\leq 2 \leq q$, and let $t\geq 0$ be such $e^{-2t\sqrt{n}} = \frac{p-1}{q-1}$. Choose  $t_{1}, t_{2}\geq 0$ so that  $t=t_{1}+t_{2}$ and $e^{-2t_{1} \sqrt{n}} = p-1$ and $e^{-2t_{2}\sqrt{n}} = \frac{1}{q-1}$. Then we have 
\begin{align*}
\|e^{-t\sqrt{-\Delta}} f\|_{q} = \|e^{-t_{2}\sqrt{-\Delta}} (e^{-t_{1}\sqrt{-\Delta}} f)\|_{q} \leq \|e^{-t_{1}\sqrt{-\Delta}} f\|_{2} \leq \|f\|_{p}.
\end{align*}

In what follows we assume $2 \leq p\leq q$. We will use a standard argument to deduce the validity of the hypercontractivity estimate from a log Sobolev inequality. Nonnegativity of the kernel for the Poisson semigroup combined with the triangle inequality implies $|e^{-t\sqrt{-
\Delta}} f| \leq e^{-t\sqrt{-\Delta}} |f|$ for any $f$. Thus  by continuity  and  standard density arguments we can assume that $f\geq 0$, $f$ is not identically zero, and $f$ is smooth in $(\ref{hyp22})$.

The equality $e^{-2t \sqrt{n}} = \frac{p-1}{q-1}$ implies $q = 1+e^{2t\sqrt{n}}(p-1)$. Fix $p\geq 2$ and consider the map
\begin{align*}
\varphi(t) = \| e^{-t\sqrt{-\Delta}} f\|_{q(t)}>0, \quad t \geq 0,
\end{align*}
where $q(t) = 1+e^{2t\sqrt{n}}(p-1)$. If we show  $\varphi'(t) \leq  0$, then we obtain $\varphi(t) \leq \varphi(0) = \|f\|_{p}$, and this proves the sufficiency part. Let $\psi(t) = \ln \varphi(t)$. We have 
\begin{align*}
\frac{q^{2}}{q'}\psi'(t) = -\ln \left(\int_{\mathbb{S}^{n}} (e^{-t\sqrt{-\Delta}} f)^{q} d\sigma_{n}\right) + \frac{\int_{\mathbb{S}^{n}} (e^{-t\sqrt{-\Delta}}f)^{q} \left( \ln (e^{-t\sqrt{-\Delta}}f)^{q} + \frac{q^{2}}{q'} \frac{\partial_{t} e^{-t\sqrt{-\Delta}}f}{e^{-t\sqrt{-\Delta}}f} \right) d\sigma_{n}}{\int_{\mathbb{S}^{n}} (e^{-t\sqrt{-\Delta}} f)^{q} d\sigma_{n}}.
\end{align*} 
Clearly $\psi'\leq 0$ if and only if 
\begin{align*}
&\int_{\mathbb{S}^{n}} (e^{-t\sqrt{-\Delta}}f)^{q} \ln (e^{-t\sqrt{-\Delta}}f)^{q}d\sigma_{n} - \int_{\mathbb{S}^{n}} (e^{-t\sqrt{-\Delta}} f)^{q} d\sigma_{n}\ln \left(\int_{\mathbb{S}^{n}} (e^{-t\sqrt{-\Delta}} f)^{q} d\sigma_{n}\right) \\
& \leq \frac{q^{2}}{q'} \int_{\mathbb{S}^{n}} (e^{-t\sqrt{-\Delta}}f)^{q-1} \sqrt{-\Delta} (e^{-t\sqrt{-\Delta}}f) d\sigma_{n}.
\end{align*}
Let $g = e^{-t\sqrt{-\Delta}}f\geq 0$. Then we can rewrite the previous inequality as 
\begin{align}\label{log1}
\int_{\mathbb{S}^{n}} g^{q} \ln g^{q} d\sigma_{n} - \int_{\mathbb{S}^{n}} g^{q} d\sigma_{n} \ln \left(\int_{\mathbb{S}^{n}} g^{q}d\sigma_{n}\right) \leq \frac{q^{2}}{2(q-1) \sqrt{n} } \int_{\mathbb{S}^{n}}g^{q-1} \sqrt{-\Delta} g d\sigma_{n},
\end{align}
where we used the fact that $q' = 2(q-1)\sqrt{n}$. Since $e^{-t\sqrt{-\Delta}}$ is contractive in $L^{\infty}(\mathbb{S}^{n})$ with a nonnegative, symmetric kernel, it follows that the validity of the  estimate  (\ref{log1}) for $q=2$ implies (\ref{log1}) for all $q \in [2, \infty)$; see, e.g., Theorem 4.1 in \cite{LG1}.

Let  $g = \sum_{k\geq 0} H_{d}$ be the decomposition of $g$ into its spherical harmonics. Then the estimate (\ref{log1}) for $q=2$ takes the form 
\begin{align*}
\int_{\mathbb{S}^{n}} g^{2} \ln g^{2} d\sigma_{n} - \int_{\mathbb{S}^{n}} g^{2} d\sigma_{n} \ln \left(\int_{\mathbb{S}^{n}} g^{2}d\sigma_{n}\right) \leq  \sum_{k\geq 0} 2\sqrt{\frac{k(k+n-1)}{n}} \, \| H_{k}\|_{2}^{2} . 
\end{align*}

It follows from Beckner's conformal log Sobolev inequality \cite{Be1} (which is a consequence of Lieb's sharp Hardy--Littlewood--Sobolev inequality \cite{Li}) that for any smooth nonnegative $g  = \sum_{k\geq 0} H_{k}$ we have 
\begin{align*}
\int_{\mathbb{S}^{n}} g^{2} \ln g^{2} d\sigma_{n} - \int_{\mathbb{S}^{n}} g^{2} d\sigma_{n} \ln \left(\int_{\mathbb{S}^{n}} g^{2}d\sigma_{n}\right) \leq  \sum_{k\geq 0} \Delta_{n}(k)\, \| H_{k}\|_{2}^{2}
\end{align*}
with $\Delta_{n}(k) = 2n \sum_{m=0}^{k-1} \frac{1}{2m+n}$. Thus, the estimate (\ref{hyp22}) is a consequence of the following lemma.

\begin{lemma}
Let $n \in \{2,3\}$. Then for all integers $k\geq 1$ one has
\begin{align*}
n \sum_{m=0}^{k-1} \frac{1}{2m+n} \leq \sqrt{\frac{k(k+n-1)}{n}}.
\end{align*}
\end{lemma}

\begin{proof}
	We first check the inequality for $k\leq 3$ by direct computation. Indeed, the case $k=1$ is an equality. The case $k=2$ can be checked as follows,
 \begin{align*}
1 + \frac{n}{2+n} = \frac{2+2n}{2+n}\leq \sqrt{\frac{2+2n}{n}}, 
 \end{align*}
 which is true because $n(2+2n) \leq (2+n)^{2}$ holds for $n=2,3$. The case $k=3$ can be checked similarly: 
 \begin{align*}
 \frac{2+2n}{2+n}+ \frac{n}{4+n} \leq \sqrt{\frac{6+3n}{n}}
 \end{align*}
 holds for $n=2,3$ (notice that this inequality fails for $n=4$). 
 
 Next, we assume $k\geq 4$. We have
 \begin{align*}
 \sum_{m=0}^{k-1} \frac{1}{m+\frac{n}{2}} =  \frac{2}{n} + \sum_{m=1}^{k-1}\frac{1}{m+\frac{n}{2}} \leq \frac{2}{n} + \int_{0}^{k-1} \frac{1}{x+\frac{n}{2}}dx =\frac{2}{n} + \ln\left(\frac{k+\frac{n}{2}-1}{\frac{n}{2}}\right).
 \end{align*}
 Thus it suffices to show 
 \begin{align*}
\frac{2}{n} + \ln\left(\frac{k+\frac{n}{2}-1}{\frac{n}{2}}\right) - \frac{2}{n} \sqrt{\frac{k(k+n-1)}{n}} \leq 0.
 \end{align*}
 Notice that the left hand side, call it $h(k)$, is decreasing in $k$. Indeed, we have  
 \begin{align*}
 h'(k) = \frac{1}{\frac{n}{2}+k-1} - \frac{2k+n-1}{n\sqrt{kn(k+n-1)}} \leq \frac{1}{\frac{n}{2}+k-1} - \frac{1}{\sqrt{kn}} \leq \frac{1}{2 \sqrt{\frac{n}{2}(k-1)}} - \frac{1}{\sqrt{kn}} \leq 0. 
 \end{align*}
 On the other hand,  we have for $n=2,3$,
 \begin{align*}
 h(4) = \frac{2}{n} + \ln \left(\frac{6+n}{n}\right) - \frac{2}{n} \sqrt{\frac{12+4n}{n}} \leq 0.
 \end{align*}
Indeed, if $n=2$, $h(4) = 1+2\ln 2 -\sqrt{10}<0$, and if $n=3$, $h(4) = \frac{2+3\ln 3 - 4\sqrt{2}}{3}<0$. 
\end{proof}

\subsection{Counterexample to $\textup{(ii)} \Rightarrow \textup{(i)}$ in high dimensions}  
Let $\lambda := \frac{n-1}{2}$, and let $C_{d}^{(\lambda)}(x)$ be the Gegenbauer polynomial 
\begin{align}\label{geg1}
C_{d}^{(\lambda)}(x) = \sum_{j=0}^{\lfloor \frac{d}{2}\rfloor} (-1)^{j} \frac{\Gamma(d-j+\lambda)}{\Gamma(\lambda) j! (d-2j)!} (2x)^{d-2j},
\end{align}
where $\lfloor \frac{d}{2} \rfloor$ denotes the largest integer $m$ such that $m \leq  \frac{d}{2}$, and $\Gamma(x)$ is the Gamma function.
Notice that if we let $Y_{d}(\xi) = C_{d}^{(\lambda)}(\xi \cdot e_{1})$, where $e_{1} = (1, 0, \ldots, 0) \in \mathbb{R}^{n+1}$, then $Y_{d}(\xi)$ is a spherical harmonic of degree $d$ on $\mathbb{S}^{n}$. In particular, for $t\geq 0$ such that $e^{-2t\sqrt{n}} = \frac{p-1}{q-1}$,  the estimate $\|e^{-t \sqrt{-\Delta}} f\|_{L^{q}(\mathbb{S}^{n})} \leq \|f\|_{L^{p}(\mathbb{S}^{n})}$ applied to $f=Y_{d}(\xi)$ is equivalent to the estimate 
\begin{align}\label{count1}
 \frac{\| Y_{d}\|_{q}}{\|Y_{d}\|_{p}} \leq e^{t \sqrt{d(d+n-1)}} = \left(\frac{q-1}{p-1}\right)^{\frac{1}{2} \sqrt{\frac{d(d+n-1)}{n}}}.
\end{align}
Next, we need
\begin{lemma}
For any $d\geq 0$ we have 
\begin{align}\label{zgvari}
\lim_{n \to \infty} \,  \frac{\| Y_{d}\|_{L^{q}(\mathbb{S}^{n}, d\sigma_{n})}}{\|Y_{d}\|_{L^{p}(\mathbb{S}^{n}, d\sigma_{n})}} = \frac{\|h_{d}\|_{L^{q}(\mathbb{R}, d\gamma)}}{\|h_{d}\|_{L^{p}(\mathbb{R}, d\gamma)}}, 
\end{align}
where $d\gamma(y) = \frac{e^{-y^{2}/2}}{\sqrt{2\pi}}dy$ is the standard Gaussian measure on the real line, and 
 $h_{d}(x)$ is the probabilistic Hermite polynomial 
 \begin{align}\label{herm1}
h_{d}(x) =  \sum_{j=0}^{\lfloor\frac{d}{2}\rfloor} \frac{(-1)^{j}d! }{j! (d-2j)!} \frac{x^{d-2j}}{2^{j}}.
 \end{align}
 \end{lemma}
 \begin{proof}
Indeed,  notice that 
\begin{align}\label{bol2}
\|Y_{d}\|_{p}^{p} = \int_{\mathbb{S}^{n}}|C_{d}^{(\lambda)}(\xi \cdot e_{1})|^{p} d\sigma_{n}(\xi)= \int_{-1}^{1}|C_{d}^{(\lambda)}(t)|^{p} c_{\lambda} (1-t^{2})^{\lambda- \frac{1}{2}} dt,
 \end{align}
 where $c_{\lambda} = \frac{\Gamma(\lambda+1)}{\Gamma(\frac{1}{2}) \Gamma(\lambda+\frac{1}{2})}$. In  particular, after the change of variables $t = \frac{s}{\sqrt{2\lambda}}$ in (\ref{bol2}), and multiplying both sides in (\ref{bol2}) by $(d!/(2\lambda)^{d/2})^{p}$ we obtain 
 \begin{align*}
 \left(\frac{d!}{(2\lambda)^{d/2}} \right)^{p}  \|Y_{d}\|_{p}^{p} = \int_{\mathbb{R}}\left| \frac{d!}{(2\lambda)^{d/2}} C_{d}^{(\lambda)}\left( \frac{s}{\sqrt{2\lambda}}\right)\right|^{p} \frac{c_{\lambda}}{\sqrt{2\lambda}} \left(1-\frac{s^{2}}{2\lambda}\right)^{\lambda- \frac{1}{2}} \mathbbm{1}_{[-\sqrt{2\lambda}, \sqrt{2\lambda}]}(s)ds, 
 \end{align*}
 where $\mathbbm{1}_{[-\sqrt{2\lambda}, \sqrt{2\lambda}]}(s)$ denotes the indicator function of the set $[-\sqrt{2\lambda}, \sqrt{2\lambda}]$.
 Notice that by Stirling's formula for any $j\geq 0$, and any $d\geq 0$  we have 
 \begin{align}\label{pred}
  \lim_{\lambda \to \infty}\,  \frac{1}{\lambda^{d-j}}\frac{\Gamma(d-j+\lambda)}{\Gamma(\lambda)} = 1.
 \end{align}
 Therefore, (\ref{herm1}) and (\ref{geg1}) together with (\ref{pred}) imply that for all $s \in \mathbb{R}$ we have
 \begin{align*}
 \lim_{\lambda \to \infty}  \frac{d!}{(2\lambda)^{d/2}} C_{d}^{(\lambda)}\left( \frac{s}{\sqrt{2\lambda}} \right) = h_{d}(s).
 \end{align*}
 Invoking Stirling's formula again we have 
 \begin{align*}
 \lim_{\lambda \to \infty} \frac{c_{\lambda}}{\sqrt{2\lambda}} \left(1-\frac{s^{2}}{2\lambda}\right)^{\lambda- \frac{1}{2}}\mathbbm{1}_{[-\sqrt{2\lambda}, \sqrt{2\lambda}]}(s) = \frac{e^{-s^{2}/2}}{\sqrt{2\pi}} \quad \text{for all} \quad s \in \mathbb{R}.
 \end{align*}
 Finally, to apply  Lebesgue's dominated convergence theorem  it suffices to verify that for all $s \in \mathbb{R}$ and all $\lambda\geq\lambda_0$ we have the following pointwise estimates 
 \begin{align*}
 &a) \quad \frac{c_{\lambda}}{\sqrt{2\lambda}} \left(1-\frac{s^{2}}{2\lambda}\right)^{\lambda- \frac{1}{2}}\mathbbm{1}_{[-\sqrt{2\lambda}, \sqrt{2\lambda}]}(s)  \leq C e^{-s^{2}/2}\\
 &b) \quad \frac{d!}{(2\lambda)^{d/2}} C_{d}^{(\lambda)}\left( \frac{s}{\sqrt{2\lambda}}\right) \leq c_{1}(d) (1+|s|)^{c_{2}(d)},
\end{align*}
where $\lambda_{0}, C,  c_{1}(d), c_{2}(d)$ are some positive  constants independent of $\lambda$ and $s$.  

To verify a) it suffices to consider the case $s \in [-\sqrt{2\lambda}, \sqrt{2\lambda}]$. Since $\lim_{\lambda \to \infty}\frac{c_{\lambda}}{\sqrt{2\lambda}} = \frac{1}{\sqrt{2\pi}}$ it follows that $\frac{c_{\lambda}}{\sqrt{2\lambda}} \leq C$ for all $\lambda \geq \lambda_{0}$, where $\lambda_{0}$ is a sufficiently large number. Next, the estimate $(1-\frac{s^{2}}{2\lambda})^{\lambda -1/2} \leq C' e^{-s^{2}/2}$ for $s \in [-\sqrt{2\lambda}, \sqrt{2\lambda}]$ follows if we show that $(1-\frac{1}{2\lambda})\ln(1-t) \leq C''/\lambda -t$ for all $t := \frac{s^{2}}{2\lambda} \in [0,1]$ where $C''$ is a universal positive constant. The latter inequality follows from $\ln (1-t) \leq -t$ for $t \in [0,1]$. 

To verify b) it suffices to show that for all  $\lambda \geq \lambda_{0}>0$ and all integers $j$ such that $d\geq j \geq 0$ one has
\begin{align*}
 \frac{1}{\lambda^{d-j}} \frac{\Gamma(d-j+\lambda)}{\Gamma(\lambda)} \leq C(d-j),
\end{align*}
where $C(d-j)$ depends only on $d-j$. The latter inequality follows from (\ref{pred}) provided that $\lambda \geq \lambda_{0}$ where $\lambda_{0}$ is a sufficiently large number. 

Thus, it follows from the Lebesgue's dominated convergence theorem that 
$$
\lim_{n \to \infty} \frac{d!}{(n-1)^{d/2}} \|Y_{d}\|_{L^{p}(\mathbb{S}^{n}, d\sigma_{n})}  = \| h_{d}\|_{L^{p}(\mathbb{R}, d\gamma)}.
 $$
 The lemma is proved.
 \end{proof}
 
Now we fix $q> \max\{p, 2\}$  and, in order to prove the failure of $\textup{(ii)} \Rightarrow \textup{(i)}$ for all sufficiently large $n$, we argue by contradiction and assume that there is a sequence of dimensions $\{n_{j}\}_{j \geq 1}$ going to infinity such that $\textup{(ii)} \Rightarrow \textup{(i)}$ in Theorem~\ref{mth} does hold. Then, by combining (\ref{count1}) and (\ref{zgvari}) we have 
\begin{align}
	\label{eq:contaass}
	\frac{\|h_{d}\|_{L^{q}(\mathbb{R}, d\gamma)}}{\|h_{d}\|_{L^{p}(\mathbb{R}, d\gamma)}} \leq \left(\frac{q-1}{p-1}\right)^{\frac{\sqrt{d}}{2}}.
\end{align}
On the other hand, a consequence of the main result in \cite{LAC} and the assumption $q>\max\{p, 2\}$ is that
 \begin{align*}
 \lim_{d \to \infty} \left(\frac{\|h_{d}\|_{L^{q}(\mathbb{R}, d\gamma)}}{\|h_{d}\|_{L^{p}(\mathbb{R}, d\gamma)}} \right)^{1/d} = \left(\frac{q-1}{\max\{p,2\}-1}\right)^{\frac{1}{2}},
 \end{align*}
which is in contradiction with (\ref{eq:contaass}).

 \begin{remark}
 Let $B(x,y)$ be the Beta function.  The estimate (\ref{count1}) for $p=2$  and $q=4$  takes the form
 \begin{align}\label{utol1}
 \int_{-1}^{1}|C_{d}^{(\frac{n-1}{2})}(t)|^{4} (1-t^{2})^{\frac{n-2}{2}} dt  \leq 9^{\sqrt{ \frac{d(d+n-1)}{n}}}  \frac{(n-1)^{2} B(1/2, n/2)}{d^{2}(2d+n-1)^{2} B^{2}(n-1, d)},
 \end{align}
 where we used the fact that $\|Y_{d}\|^{2}_{L^{2}(\mathbb{S}^{n})} = \frac{n-1}{d(2d+n-1) B(n-1, d)}$. The numerical computations show that the inequality (\ref{utol1}) already fails for $d=7$ and $n=13$. 
\end{remark}

\subsection*{Acknowledgements} Partial support through US National Science Foundation grants DMS-1363432 and DMS-1954995 (R.L.F.) as well as DMS-2052645, DMS-1856486, and CAREER-DMS-2052865, CAREER-DMS-1945102 (P.I.) is acknowledged.


\end{document}